\documentclass[12pt]{amsart}
\usepackage{amssymb,pb-diagram}
\usepackage{amsrefs}
\usepackage{rotating}

\newtheorem{theorem}{Theorem}[section]

\newtheorem{corollary}[theorem]{Corollary}

\newtheorem{proposition}[theorem]{Proposition}
\newtheorem{question}{Question}

\theoremstyle{definition}
\newtheorem{definition}[theorem]{Definition}

\theoremstyle{remark}
\newtheorem{remark}{Remark}

\def\cdots{\mathinner{\cdotp\cdotp\cdotp}}

\def\dist{\text{dist}}

\def\o{\scr \sigma(K)}
\def\i{i_{\mathcal{CB}}(K)}

\newcommand{\e}{\varepsilon}

 \newcommand{\F}{\mathcal F}
 
 \newcommand{\M}{\mathcal M}

\newcommand{\N}{\ensuremath{\mathbb{N}}}
\newcommand{\NN}{\ensuremath{\mathcal{N}}}

\newcommand{\mk}{\medskip}

\newcommand{\isometric}{\ensuremath{\underset{=}{\lhook\joinrel\relbar\joinrel\rightarrow}}}

\newcommand{\lip}{\ensuremath{\underset{Lip}{\lhook\joinrel\relbar\joinrel\rightarrow}}}
\newcommand{\clip}{\ensuremath{\underset{C-Lip}{\lhook\joinrel\relbar\joinrel\rightarrow}}}
\newcommand{\onelip}{\ensuremath{\underset{1-Lip}{\lhook\joinrel\relbar\joinrel\rightarrow}}}
\newcommand{\uptwolip}{\begin{sideways}$\ensuremath{\underset{2-\textrm{Lip}}{\lhook\joinrel\relbar\joinrel\relbar\joinrel\relbar\joinrel\relbar\joinrel\rightarrow}}$\end{sideways}}
\newcommand{\updonelip}{\begin{sideways}$\ensuremath{\underset{D_1-\textrm{Lip}}{\lhook\joinrel\relbar\joinrel\relbar\joinrel\relbar\joinrel\relbar\joinrel\rightarrow}}$\end{sideways}}
\newcommand{\updklip}{\begin{sideways}$\ensuremath{\underset{D_k-\textrm{Lip}}{\lhook\joinrel\relbar\joinrel\relbar\joinrel\relbar\joinrel\relbar\joinrel\rightarrow}}$\end{sideways}}
\newcommand{\updwlip}{\begin{sideways}$\ensuremath{\underset{D_{\omega}-\textrm{Lip}}{\lhook\joinrel\relbar\joinrel\relbar\joinrel\relbar\joinrel\relbar\joinrel\rightarrow}}$\end{sideways}}
\newcommand{\upisometric}{\begin{sideways}$\ensuremath{\underset{=}{\lhook\joinrel\relbar\joinrel\relbar\joinrel\relbar\joinrel\rightarrow}}$\end{sideways}}

\newcommand{\almost}{\ensuremath{\underset{a.i.}{\lhook\joinrel\relbar\joinrel\rightarrow}}}

\begin{document}

\title[On the $\scr C(K)$-distortion of classes of metric spaces]
{A topological obstruction for small-distortion embeddability into spaces of continuous functions on countable compact metric spaces}
\author{Florent Baudier}
\address{Texas A\&M University}
\email{florent@math.tamu.edu}


\begin{abstract} We give the first lower bound on the $\scr C(K)$-distortion of the class of separable Banach spaces, for $K$ a countable compact
in the family $\{ [0,\omega],[0,\omega\cdot2],\cdots, [0,\omega^2], \cdots, [0,\omega^k\cdot n],\cdots,[0,\omega^\omega]\}$.
\end{abstract}

\maketitle{}        

\section{introduction}

\subsection{Motivation}\ \\
In this paper we study a natural question from the Nonlinear Geometry of Banach spaces, namely: What is the best possible distortion when embedding bi-Lipschitzly a separable Banach space into a space $\scr C(K)$, the Banach space of continuous functions on some countable compact metric space $K$? Our motivation is to compute the best distortion allowed when embedding {\it all} separable Banach or metric spaces. We recall briefly a few classical embedding results involving $\scr C(K)$-spaces. Back in 1906, Fr\'echet observed \cite{Frechet1906} that every separable metric space admits an isometric embedding into the space $\ell_\infty(\N)$. An easy application of the Hahn-Banach theorem gives a linear isometric embedding of every separable Banach space into $\ell_\infty(\N)$. This result can be cast as an embedding result into a $\scr C(K)$-space since $\ell_\infty(\N)$ is actually the space $\scr C(\beta\N)$ where $\beta\N$ denotes the Stone-Cech compactification of $\N$. $\beta\N$ is an uncountable compact space. $\ell_\infty(\N)$ being non separable $\beta\N$ cannot be metrizable.The Banach-Mazur theorem \cite{BanachMazur33} asserts that every separable Banach space admits an isometric embedding into the space $\scr C([0,1])$. Note that $[0,1]$ equipped with its canonical distance is an uncountable metric compact.
With the help of Fr\'echet's embedding it is easily seen that every separable metric space can be isometrically embedded into $\scr C([0,1])$.
In 1974 in \cite{Aharoni1974}, Aharoni proved that the $c_0^+$-distortion of every separable metric space is less than $6$. In the same paper he also proved that the $c_0$-distortion of $\ell_1$ is at least $2$. A few years later Assouad \cite{Assouad1978} showed that $c_0^+$-distortion of every separable metric space is $3$. The fact that there is an bi-Lipschitz embedding  with distortion exactly $3$ and that this value is optimal for embeddings into $c_0^+$ is due to Pelant \cite{Pelant1994}. Finally the end of the story regarding embedding into $c_0$ was settled when Kalton and Lancien \cite{KaltonLancien2008} recently constructed an embedding with distortion $2$ (resp. 1) for every separable (resp. proper) metric space. The theory discussed so far is clearly isometric and despite $c_0$ is not isometric to a $\scr C(K)$-space for any compact $K$, embeddings into $c_0$ are related to those into $\scr C(K)$-spaces. Indeed $c_0$ is an hyperplan of the space of convergent sequences, sometimes denoted $c$, which can be seen as the space $\scr C(K)$
where $K=\gamma\N$ the Alexandrov-compactification of $\N$,  also called the one-point compactification of $\N$. Moreover it is easy to show that whenever $K$ is an infinite (not necessarily metric) compact Hausdorff space, $\scr C(K)$ contains a subspace isometric to $c_0$ (see \cite{AlbiacKalton2006} Proposition 4.3.11). 

\mk

From now on we will essentially consider countable compact metric spaces. Every countable compact metric space can be identify with a closed interval of ordinal numbers equipped with the order topology. As usual an ordinal $\beta$ is identified with the interval $[0,\beta)=\{\alpha; \alpha<\beta\}$ and $\beta+1$ with the compact $[0,\beta]$. For instance $c=\scr C([0,\omega])$, $\omega$ being the first infinite countable ordinal. We will consider the nested family of countable compact intervals: $$[0,\omega]\subset[0,\omega\cdot 2]\subset\cdots\subset[0,\omega^2]\subset\cdots\subset[0,\omega^n\cdot k]\subset\cdots\subset[0,\omega^\omega]$$

It follows from a theorem of Borsuk \cite{Borsuk1933} that the collection of $\scr C(K)$-spaces associated to this nested collection has the property that $\scr C(K_1)$ embeds linearly isometrically into $\scr C(K_2)$ whenever $K_1\subset K_2$. Our previous discussion can be summarized in the following self-explanatory diagram where $(M,d)$ (resp. $(X,\|\cdot\|)$) denotes any separable metric space (resp. separable Banach space):
\begin{tiny}
$$\begin{array}{ccccccccc}
c_0&\isometric&\scr C([0,\omega])&\isometric\cdots\isometric&\scr C([0,\omega^k])&\isometric\cdots\isometric&\scr C([0,\omega^\omega])&\isometric &\scr C([0,1])\\
\uptwolip& &\updonelip & & \updklip & & \updwlip& &\upisometric\\
(M,d)      & & (M,d)      &  &   (M,d)    & & (M,d)     & &(X,\|\cdot\|)\\
\end{array}$$
\end{tiny}

Whereas the best distortion achievable in the two extreme cases is completely understood as of now, essentially no estimates for the values of the distortions $D_1,\cdots, D_k,\cdots, D_\omega$ are known besides the upper bound $2$ which can be deduced from Kalton and Lancien embedding. It is worth noting that since subspaces of $\scr C(K)$ (for $K$ countable) are $c_0$-saturated, the subspace structure of $c_0$ allows us to conclude that $\scr C(K)$ cannot be a {\it linearly} isometrically universal space for the class of separable Banach spaces. And it cannot be an isometrically universal space either since Godefroy and Kalton
\cite{GodefroyKalton2003} proved that if a separable Banach space $X$ embeds isometrically into a Banach space $Y$, then $Y$ contains an isometric linear copy of $X$. In this paper we will prove that $D_k\ge\frac{k+1}{k}$ (in particular $D_2=2$) and we will discuss the case of $D_\omega$.

\mk

To estimate $D_k$ from below we exhibit a connection between a topological property of the compact $K$ and the $\scr C(K)$-distortion of a particular family of separable metric spaces. Roughly speaking if the compact is too small then the $\scr C(K)$-distortion of those metric spaces cannot be too small either. The notion of smallness that we use is related to the size of the Cantor-Bendixson derivatives of the compact. In the next two short sections we fix the notation to be used.

\subsection{$\scr C(K)$-distortion of fundamental classes of metric spaces}
Let $\M$ and $\NN$ be two metric spaces. Define the {\it distortion} of $f\colon \M\to \NN$ to be
$$ \dist(f):= \|f\|_{Lip}\|f^{-1}\|_{Lip}=\sup_{x\neq y \in
\M}\frac{d_\NN(f(x),f(y))}{d_\M(x,y)}.\sup_{x\neq y \in
\M}\frac{d_\M(x,y)}{d_\NN(f(x),f(y))}.$$ If the distortion of $f$ is finite, $f$ is said to be a {\it bi-Lipschitz embedding}. The notation $\M\lip \NN$ means that $\M$ bi-Lipschitzly embeds into $\NN$. If dist$(f)\leq C$, one will use the notation $\M \clip \NN$. Denote 

$$c_{\NN}(\M)=\inf\{\dist(f); f\colon \M\lip \NN\}$$

Let $\F$ be a collection of metric spaces and define

$$c_{\NN}(\F)=\sup\{c_{\NN}(\M); \M\in \F\}$$

In this paper the parameter $c_{\NN}(\F)$, called the $\NN$-distortion of the class $\F$, is studied for the following classes and spaces:

\begin{itemize}

\item $\F$ is either the class $\mathcal{SEP}$ or $\textrm{SEP}$, where 

$$\mathcal{SEP}:=\{\M; \M\textrm{ separable metric space}\}$$

$$\textrm{SEP}:=\{X; X\textrm{ separable Banach space}\}$$

\item $\NN$ is $\scr C(K)$ for some infinite metric compact $K$.
\end{itemize} 
\subsection{Cantor-Bendixson derivation}

\begin{definition}
Let $K$ be a compact topological space and define $K'$, the Cantor-Bendixson derivative of $K$ by:
$$K'=K\backslash\left\{x\in K: x\ \textrm{is an isolated point in}\ K\right\}.$$ 
$K'$ is the subset of all accumulation points of $K$.
\end{definition}

\begin{remark}
If $K$ is metrisable then $$K'=\left\{x\in K:\exists (x_n)_{n\in\N}\subset K\setminus\{x\}\textrm{ such that }x_n\to_{n\to\infty} x\right\}.$$
\end{remark}

By transfinite induction one can define derivatives of higher order.

\medskip

\begin{itemize}
\item $K^{(0)}=K$, $K^{(1)}=K'$\\
\item if $\alpha$ is a successor ordinal ($\alpha=\beta+1$), $K^{(\beta+1)}=(K^{(\beta)})'$,\\
\item if $\alpha$ is a limit ordinal, $K^{(\alpha)}=\bigcap_{\beta<\alpha}K^{(\beta)}$.\\
\end{itemize}

Let us briefly recall some classical properties of the Cantor-Bendixson derivation.
\begin{proposition}
Let $K$ be a compact metric space, then
\begin{enumerate}
 \item $K$ is finite $\Longleftrightarrow$ $K'=\emptyset$ $\Longleftrightarrow$ $K$ is discrete.
 \item $K$ is countable $\Longleftrightarrow$ $\exists\ \alpha<\omega_1$ such that $K^{(\alpha)}=\emptyset$.
 \item $K$ is uncountable $\Longleftrightarrow$ $\exists\ \alpha<\omega_1$ such that $K^{(\alpha+1)}=K^{(\alpha)}\neq\emptyset$.
\end{enumerate}
\end{proposition}

\section{Estimating the $\scr C(K)$-distortion from below}
\subsection{A lower bound on the $\scr C(K)$-distortion for countable metric spaces}
Aharoni proved that $c_{c_0}(\textrm{SEP})\ge 2$, hence $c_{c_0}(\scr{MET})\ge 2$. Indeed he showed that the separable Banach space $\ell_1$ does not embed into $c_0$ with distortion strictly less than $2$. 
Actually a careful inspection of his proof shows that the proof and the statement of the result can be carried out and stated without mentioning and using the linear structure of the Banach space $\ell_1$. This simple but crucial observation allows us to extend Aharoni's proof to the much more general setting of embeddings into $\scr C(K)$-spaces. 

\mk

For $k\ge 0$, denote $\N^{[k]}$ the set of all subsets of $\N$ with exactly $k$ elements, and $\N^{\le k}=\bigcup_{i=0}^k \N^{[k]}$. By convention $\N^{[0]}=\emptyset$. Equip $\N^{\le k}$ with the symmetric difference metric $d_\Delta$, i.e.  $d_\Delta(\sigma,\tau):=|\sigma\bigtriangleup \tau|$, $\sigma,\tau\in \N^{\le k}$ (the cardinality of the symmetric difference). For instance, $d_\Delta(\{1,4,6,100\},\{4,6, 33\})=3$.

\mk

One shall denote $\Delta_{\le k}(\N)$ the metric space $(\N^{\le k},d_\Delta)$. $\Delta_{\le k}(\N)$ is a $2k$-bounded, $1$-separated, countable non-proper metric space. The set of elements of $\Delta_{\le k}(\N)$ can be identified with the rooted countably branching tree of height $k$, whose root is the empty set. However the metric $d_\Delta$ is not the classical tree metric with which the countably branching tree is usually endowed. Denote by $T_{1,2}^2(\N)$ the subset of the countably branching tree of height $2$ consisting of the first two elements of the first level and their successors. Formally $T_{1,2}^2=\{\{1\},\{2\},\{1,n\},\{2,m\}; n,m\in\N\}$. The following theorem is nothing but Aharoni's lower bound theorem reformulated in purely metric terms. We include the original proof using our notation for the sake of completeness and hoping that it will make the notation used in the proof of Theorem \ref{index} more accessible.  
\begin{theorem}[Aharoni]
The metric space $T_{1,2}^2(\N)$ does not embed into $c_0$ with distortion strictly less than $2$.
\end{theorem}

\begin{proof} 
Let $f:T_{1,2}^2(\N)\to c_0$ satisfying $$d_\Delta(\sigma,\tau)\le\Vert f(x)-f(y)\Vert_{\infty}\le Cd_\Delta(\sigma,\tau),\hskip1cm \sigma,\tau\in T_{1,2}^2(\N)$$
\noindent Without loss of generality one can assume that $f(\emptyset)=0$. If $C<2$, define for every $i\neq j$, $$\scr X_{i,j}:=\{n\in\N : \vert f(\{i\})_n-f(\{j\})_n\vert\ge 4-2C\}.$$

$\scr X_{i,j}$ are finite sets. Moreover, for every $i,j\ge 3$, $i\neq j$, $\scr X_{1,2}\bigcap \scr X_{i,j}\neq \emptyset$. Indeed, $$\Vert f(\{1,i\})-f(\{2,j\})\Vert_{\infty}\ge d_\Delta(\{1,i\},\{2,j\})=4$$
Hence there exists $n_{i,j}\in\N$ such that $$\vert f(\{1,i\})_{n_{i,j}}-f(\{2,j\})_{n_{i,j}}\vert\ge 4.$$ Therefore, $$\begin{array}{rl}
\vert f(\{i\})_{n_{i,j}}-f(\{j\})_{n_{i,j}}\vert \ge &\vert f(\{1,i\})_{n_{i,j}}-f(\{2,j\})_{n_{i,j}}\vert-\vert f(\{1,i\})_{n_{i,j}}-f(\{i\})_{n_{i,j}}\vert\\
 & -\vert f(\{2,j\})_{n_{i,j}}-f(\{j\})_{n_{i,j}}\vert\\
 & \\ 
\ge & 4-\Vert f(\{1,i\})-f(\{i\})\Vert_{\infty}-\Vert f(\{2,j\})-f(\{j\})\Vert_{\infty}\\
&  \\
 \ge & 4-Cd_\Delta(\{1,i\},\{i\})-Cd_\Delta(\{2,j\},\{j\})=4-2C.\\
\end{array}$$
This proves that $n_{i,j}\in \scr X_{i,j}$. Arguing along the same lines one gets that $n_{i,j}\in \scr X_{1,2}$ as well. Therefore $\scr X_{1,2}\bigcap \scr X_{i,j}\neq\emptyset$ whenever $i\neq j$, $i,j\ge 3$. Denote $P$ the canonical projection from $c_0$ onto the closed linear span generated by the vectors $(e_n)_{n\in \scr X_{1,2}}$. The sequence $\left(P f(\{n\})\right)_{n=3}^{\infty}$ is a $C$-bounded and $(4-2C)$-separated sequence in a finite-dimensional Banach space; a contradiction. Indeed, 
\mk

\noindent for every $n\ge 3$,

\begin{align*}
\|P f(\{n\})\|_\infty\le\|f(\{n\})\|_\infty=\|f(\{n\})-f(\emptyset)\|_\infty\le Cd_\Delta(\{n\},\emptyset)=C
\end{align*}
and for every $i\neq j$, $i,j\ge 3$

\begin{align*}
\|P f(\{i\})-P f(\{j\})\|_\infty& =\sup_{n\in \scr X_{1,2}}|f(\{i\})_n-f(\{j\})_n|\\
&\ge |f(\{i\})_{n_{i,j}}-f(\{j\})_{n_{i,j}}|\\
& \ge 4-2C>0\\
\end{align*}

\end{proof}

Inspired by the reformulation of Aharoni's proof in terms of a metric subset of $\Delta_{\le 2}(\N)$ we establish a link between the $\scr C(K)$-distortion of the sequence $(\Delta_{\le k}(\N))_{k\ge 1}$ and the Cantor-Bendixson index of the compact $K$.
 
\begin{theorem}\label{index} Let $K$ be a compact space and an integer $k\ge 2$.\\
If $\Delta_{\le k}(\N)$ admits a bi-Lipschitz embedding into $\scr C(K)$ with distortion strictly less than $\frac{k}{k-1}$ then $K^{(k-1)}$ is infinite.
\end{theorem}

\begin{proof}

Let $f\colon \Delta_{\le k}(\N)\to \scr C(K)$ such that $$d_\Delta(\sigma,\tau)\le\Vert f(\sigma)-f(\tau)\Vert_{\infty}\le Dd_\Delta(\sigma,\tau),\ \sigma, \tau\in \Delta_{\le k}(\N)$$ One can assume that $f({\emptyset})=0$. Let $(n_i^1)_{i\ge 1}, \cdots, (n_i^k)_{i\ge 1}$ be $k$ sequences of mutually distinct elements in $\N$, i.e. $n_i^r=n_j^s$ if and only if $i=j$ and $r=s$ (for instance take $n_i^r=k(i-1) + r-1$). Remark that $\eta:=2k-2(k-1)D>0$ if $D<\frac{k}{k-1}$ and define for $1\le r\le k$ and $i,j\ge 1$

$$\scr X_{i,j}^r:=\left\{ \beta\in K: \vert f(\{n_i^r\})(\beta)-f(\{n_j^r\})(\beta)\vert\ge \eta\right\}$$
 
 \mk
 
\noindent CLAIM 1: For all $i_1\neq j_1,\cdots, i_k\neq j_k$ one has $\scr{X}_{i_1,j_1}^1\bigcap\cdots\bigcap\scr{X}_{i_k,j_k}^k\neq\emptyset$ 

\mk

\noindent Proof of CLAIM 1:
Denote $\sigma=\{n_{i_1}^1, \cdots,n_{i_k}^k\}$ and $\tau=\{n_{j_1}^1, \cdots,n_{j_k}^k\}$. When $i_1\neq j_1,\cdots, i_k\neq j_k$, $\sigma\Delta\tau=\emptyset$ hence $\|f(\sigma)-f(\tau)\|_{\infty}\ge 2k$ and there exists $\beta\in K$ such that

\begin{equation}\label{e1}
  \vert f(\sigma)(\beta)-f(\tau)(\beta)\vert\ge 4                                                         
\end{equation}

and for all $1\le r\le k$

\begin{equation}\label{e2}
     |f(\{n_{i_r}^r\})(\beta)-f(\{n_{j_r}^r\})(\beta)|\ge \eta>0.
    \end{equation}

Indeed, \begin{align*}
|f(\{n_{i_r}^r\})(\beta)-f(\{n_{j_r}^r\})(\beta)| \ge &| f(\sigma)(\beta)-f(\tau)(\beta)|-| f(\{n_{i_r}^r\})(\beta)-f(\sigma)(\beta)|\\
 & -\vert f(\tau)(\beta)-f(\{n_{j_r}^r\})(\beta)|\\
 & \\ 
\ge & 2k-\|f(\{n_{i_r}^r\})-f(\sigma)\|_{\infty}-\| f(\{n_{j_r}^r\})-f(\tau)\|_{\infty}\\
 &  \\
 \ge & 2k-Dd_\Delta(\{n_{i_r}^r\},\sigma)-Dd_\Delta(\{n_{j_r}^r\},\tau)\\
 &  \\
 \ge & 2k-2D(k-1)=\eta>0.\\
\end{align*}
 
\noindent CLAIM 2: The set $\scr{X}_{i_1,j_1}^1\bigcap\cdots\bigcap\scr{X}_{i_{k-1},j_{k-1}}^k$ is infinite,  for all $i_1\neq j_1,\cdots, i_{k-1}\neq j_{k-1}$.

\mk

\noindent Proof of CLAIM 2: Fix $i_1\neq j_1,\cdots, i_{k-1}\neq j_{k-1}$ and denote $\scr{B}:=\scr{X}_{i_1,j_1}^1\bigcap\cdots\bigcap\scr{X}_{i_{k-1},j_{k-1}}^k$. It follows from CLAIM 1 that for every $i\neq j$ there exists an element $\beta_{i,j}\in \scr B\bigcap \scr X_{i,j}^k$. Let $\Gamma_k:=\bigcup_{i\neq j}(\scr{B}\bigcap \scr{X}_{i,j}^k)\subset K$ and consider $$\begin{array}{rcl}
   \Phi_k:\N & \to & \ell_{\infty}(\Gamma_k)\\
i & \mapsto &  \left( f(\{n_i^k\})(\gamma)\right)_{\gamma\in \Gamma_k}\\
  \end{array}$$

$(\Phi_k(i))_{i\ge1}$ is a $D$-bounded and $\eta$-separated sequence. Indeed, 

\mk

\noindent for every $i\ge 1$,

\begin{align*}
\|\Phi_k(i)\|_\infty\le\|f(\{n_i^k\})\|_\infty=\|f(\{n_i^k\})-f(\emptyset)\|_\infty\le Dd_\Delta(\{n_i^k\},\emptyset)=D
\end{align*}

and for every $i\neq j$,

\begin{align*}
\|\Phi_k(i)-\Phi_k(j)\|_\infty& =\sup_{\gamma\in \Gamma_k}|f(\{n_i^k\})(\gamma)-f(\{n_j^k\})(\gamma)|\\
&\ge |f(\{n_i^k\})(\beta_{i,j})-f(\{n_j^k\})(\beta_{i,j})|\\
& \ge \eta>0\\
\end{align*}

It follows from a classical compactness argument that $\Gamma$ must be infinite and consequently $\scr{B}\supset \Gamma$ is infinite. 

\mk

\noindent By definition, 

$$K^{'}\bigcap\scr X_{i,j}^r:=\left\{ \beta\in K^{'}: \vert f(\{n_i^r\})(\beta)-f(\{n_j^r\})(\beta)\vert\ge \eta\right\}$$

\mk

\noindent CLAIM 3: For all $i_1\neq j_1,\cdots, i_{k-1}\neq j_{k-1}$ one has $$K^{'}\bigcap\scr{X}_{i_1,j_1}^1\bigcap\cdots\bigcap\scr{X}_{i_{k-1},j_{k-1}}^k\neq\emptyset$$

\noindent Proof of CLAIM 3: Fix $i_1\neq j_1,\cdots, i_{k-1}\neq j_{k-1}$. It follows from CLAIM 2 that $\scr{B}:=\scr{X}_{i_1,j_1}^1\bigcap\cdots\bigcap\scr{X}_{i_{k-1},j_{k-1}}^k\subset K$ is infinite hence there is an accumulation point $\beta\in K^{'}$. But $\beta=\lim \beta_n$ with $\beta_n\in \scr B$ and for every $1\le r\le k-1$ since $f(\{n_{i_r}^r\})$ and $f(\{n_{j_r}^r\})$ are continuous maps on $K$, $\beta\in\scr{X}_{i_r,j_r}^r$ which proves the claim.

\mk

Repeating the previous procedure (CLAIM 1 and CLAIM 2) we can show that $K^{'}\bigcap\scr{X}_{i_1,j_1}^1\bigcap\cdots\bigcap\scr{X}_{i_{k-2},j_{k-2}}^k$ is infinite for all $i_1\neq j_1,\cdots, i_{k-2}\neq j_{k-2}$. After $k-1$ such operations we get that 
$K^{(k-1)}\bigcap\scr{X}_{i_1,j_1}^1\neq\emptyset$ for all $i_1\neq j_1$ and we conclude as follow. 

Let $\Gamma_1:=\bigcup_{i\neq j}(K^{(k-1)}\bigcap\scr{X}_{i,j}^1)\subset K^{(k-1)}$. Consider $$\begin{array}{rcl}
   \Phi_1:\N & \to & \ell_{\infty}(\Gamma_1)\\
i & \mapsto &  \left( f(n_i^1)(\gamma)\right)_{\gamma\in \Gamma_1}\\
  \end{array}$$

$(\Phi_1(i))_{i\ge 1}$ is again a sequence that is $D$-bounded and $\eta$-separated, hence $\Gamma_1$ must be infinite. We deduce that $K^{(k-1)}$ is infinite. 

\end{proof}

\subsection{Application to the $\scr C(K)$-distortion for the class of separable Banach spaces}
Let $\Delta_{\omega}(\N)$ the set of all finite subsets of $\N$ equipped with the symmetic difference distance. It is clear that $\Delta_{\omega}(\N)$ contains isometric copies of all the metric spaces $\Delta_{\le k}(\N)$. Also remark that the map $f\colon \Delta_{\omega}(\N)\to \ell_1$ such that $f(\{n_1,\cdots, n_r\})=\sum_{i=1}^r e_{n_i}$ is an isometric embedding of
$\Delta_{\omega}(\N)$ into $\ell_1$. We could now argue along the following lineS to prove that distortions strictly less than $2$ are forbidden when embedding $\ell_1$ into $c_0$. $C([0,\omega])$ contains a subspace isometric to $c_0$. Since $\ell_1$ contains an isometric copy of $\Delta_2(\N)$, it follows from Theorem \ref{index} that $\ell_1$ does not admit bi-Lipschitz embedding into $c_0$ with a distortion strictly smaller than $2$ otherwise it will embed into $C([0,\omega])$ with the same distortion.

We say that $\M$ embeds {\it almost isometrically} into $\NN$ ($\M\almost \NN$) if for every $\epsilon>0$ there exist a bi-Lipschitz embedding $f$ from $\M$ into $\NN$ with $\textrm{dist}(f)\leq 1+\epsilon$

\begin{corollary}\label{ai}
If $\Delta_{\omega}(\N)\almost C(K)$ then $K^{(\omega)}\neq\emptyset$. In particular if $\scr C(K)$ is an almost isometric universal space for the class of separable Banach spaces then $K^{(\omega)}\neq\emptyset$. 
\end{corollary}

\begin{proof}
It follows from Theorem \ref{index} that $K^{(k)}\neq\emptyset$ for every $k<\omega$ hence 
$K^{(\omega)}=\bigcap_{k<\omega} K^{(k)}\neq\emptyset$ by a classical Baire argument.
\end{proof}

The following theorem, of independent interest, can also be used to prove Corollary \ref{ai} in combination with Aharoni's original lower bound estimate. 
\begin{theorem}
If $\ell_1\almost \scr C(K)$ then for all $\alpha<\omega$, $\ell_1\almost \scr C(K^{(\alpha)})$.
\end{theorem}

\begin{proof}
Assume that $\ell_1\almost \scr C(K)$, we will prove by induction the following property: 

\mk

\noindent For all $\e>0$ and $\alpha<\omega$ there exists $ f:c_{00}\to \scr C(K^{(\alpha)})$ such that for every $x, y\in c_{00}$, $x$ and $y$ with disjoint supports, $$\frac{\Vert x-y\Vert_1}{1+\e}\le\Vert f(x)-f(y)\Vert_{\infty}\le\Vert x-y\Vert_1,$$

\noindent where $c_{00}$ is the space of sequences with finite support.

\mk

Let $\e>0$. There exists $g\colon c_{00}\to \scr C(K)$ and $\e'$, to be chosen later, such that $$\frac{\Vert x-y\Vert_1}{1+\e'}\le\Vert g(x)-g(y)\Vert_{\infty}\le\Vert x-y\Vert_1.$$ For every $x\in c_{00}$ define, $$f(x)=g(x)_{\vert K'}\in \scr C(K').$$ $K'$ being a subset of $K$ it is clear that $\Vert f(x)-f(y)\Vert_{\infty}\le\Vert x-y\Vert_1$. Let $x,y\in c_{00}$ with disjoint supports such that $\Vert x-y\Vert_1=\Vert x\Vert_1+\Vert y\Vert_1=\delta>0$, for some fixed $\delta$. For $i$ and $j$ not in the union of the supports of $x$ and $y$ we have,
\begin{equation}
       \Vert g(x+\delta e_i)-g(y+ \delta e_j)\Vert_{\infty}\ge\frac{3\delta}{1+\e'}                                              
\end{equation}

Hence there exists $\beta\in K$ such that 
\begin{equation}\label{equa1}
       \vert g(x+\delta e_i)(\beta)-g(y+ \delta e_j)(\beta)\vert\ge\frac{3\delta}{1+\e'}
\end{equation}

\begin{equation}
       \vert g(\delta e_i)(\beta)-g(\delta e_j)(\beta)\vert\ge\frac{(2-\e')\delta}{1+\e'}                                              
\end{equation}

\begin{equation}
       \vert g(x)(\beta)-g(y)(\beta)\vert\ge\frac{(1-2\e')\delta}{1+\e'}=\frac{(1-2\e')\Vert x-y\Vert_1}{1+\e'}                                              \end{equation}

Let $$B=\left\{ \beta\in K:\exists i\neq j\ \textrm{satisfying equation}\ (\ref{equa1})\right\}.$$  By a now classical compactness argument we can show that $B$ is infinite and $K'$ has an element $\beta_0$ such that, by continuity of $g(x)$ and $g(y)$ with $x,y$ fixed, 
$$\vert g(x)(\beta_0)-g(y)(\beta_0)\vert\ge \frac{(1-2\e')\Vert x-y\Vert_1}{1+\e'}.$$

and 

$$\Vert f(x)-f(y)\Vert_{\infty} \ge \frac{(1-2\e')\Vert x-y\Vert_1}{1+\e'}.$$

The property is initialized by choosing $\e'$ such that $\frac{1-2\e'}{1+\e'}>\frac{1}{1+\e}$. Assume that the property holds for an ordinal $\alpha$. Then the property holds for the ordinal $\alpha+1$.  Indeed apply the proof for $\alpha=0$ to the map $f=g_{\vert K^{(\alpha)}}$.
\end{proof}

\section{Conclusion and Questions}

The smallest ordinal $\alpha$ such that $K^{(\alpha)}=\emptyset$ (denoted $\i$ in this paper) is classically called the Cantor-Bendixson index (or rank) of $K$. According to our work it seems more relevant to consider the following index which will be called the {\it Cantor-Bendixson derivation order} of $K$. We will follow the notation from \cite{AlbiacKalton2006} and denote it as follow:
$$\o:=\sup\{\alpha; K^{(\alpha)}\neq\emptyset\}$$
Clearly $\i=\o+1$. For instance we have, $\sigma(\gamma\N)=1$, $\sigma([0,\omega^\alpha\cdot n+r])=\alpha$, $\sigma([0,1])=\infty$ and $\sigma(\beta\N)=\infty$.

\mk

\noindent Clearly $c_{\NN}(\textrm{SEP})\le c_{\NN}(\scr{MET})$ and $c_{\M}(X)\le c_{\NN}(X)$ whenever $\NN\onelip\M$. As of now the situation regarding estimates of the $\scr C(K)$-distortion of the classes of separable metric or Banach spaces is the following:  

\begin{itemize}
\item $c_{\scr C([0,1])}(\textrm{SEP})=c_{\scr C([0,1])}(\mathcal{SEP})=c_{\scr C(\beta\N)}(\textrm{SEP})=c_{\scr C(\beta\N)}(\mathcal{SEP})=1$\\
\item If $K$ is a metric compact with $\o=1$ then $$c_{\scr C(K)}(\textrm{SEP})=c_{\scr C(K)}(\scr{SEP})=2$$
\item If $K$ is a metric compact with $2\le \o<\omega$ then $$\dfrac{\o+1}{\o}\le c_{\scr C(K)}(\textrm{SEP})\le c_{\scr C(K)}(\scr{SEP})\le 2$$
\item If $c_{\scr C(K)}(\textrm{SEP})=1$ then $\o\ge\omega$. In other words if $\scr C(K)$ is an almost isometric universal space for the class of separable Banach spaces then $\o\ge\omega$.
\end{itemize}

\begin{question} What is the exact value of $c_{\scr C(K)}(\textrm{SEP})$ (or $c_{\scr C(K)}(\scr{SEP})$) when $K$ is a metric compact with $2\le \o$?
\end{question}

\begin{question} Characterize the compacts $K$ such that $c_{\scr C(K)}(\textrm{SEP})=1$ (or $c_{\scr C(K)}(\scr{SEP})=1$)? And for which ones $\scr C(K)$ is an isometric universal Banach space for the class of separable Banach spaces?
\end{question}

\end{document}